\documentclass[a4paper,12pt]{amsart}
\pdfoutput=1
\usepackage{microtype}
\usepackage{amssymb,latexsym}
\usepackage[margin=3.25cm]{geometry}
\usepackage{mathtools}
\usepackage{amsthm}
\usepackage[inline]{enumitem}
\usepackage{mathrsfs}
\usepackage{mleftright}
\usepackage[dvipsnames]{xcolor}
\usepackage{hyperref}
\hypersetup{
    colorlinks=true,
    linkcolor=black, 
    citecolor=black} 

\DeclareMathOperator{\tr}{tr}

\newcommand{\diff}{\mathop{}\!d}
\newcommand{\grad}{\mathop{}\!\nabla}
\newcommand{\tildegrad}{\mathop{}\!\widetilde{\nabla}}
\newcommand{\pitop}{\mathop{}\!\pi^{\top}}
\newcommand{\piperp}{\mathop{}\!\pi^{\perp}}

\theoremstyle{plain}
\newtheorem{theorem}{Theorem}
\newtheorem{proposition}[theorem]{Proposition}
\newtheorem{lemma}[theorem]{Lemma}
\newtheorem{corollary}[theorem]{Corollary}
\theoremstyle{remark}
\newtheorem{remark}[theorem]{Remark}
\theoremstyle{definition}

\newtheorem{definition}[theorem]{Definition}
\newtheorem*{notation*}{Notation}

\begin{document}
\title[Planar pseudo-geodesics]{Planar pseudo-geodesics\\and totally umbilic submanifolds}
\author[S. Markvorsen]{Steen Markvorsen}
\address{Department of Applied Mathematics and Computer Science\\
Technical University of Denmark\\
Matematiktorvet, Building 303B\\
2800 Kongens Lyngby\\
Denmark}
\email{stema@dtu.dk}

\author[M. Raffaelli]{Matteo Raffaelli}
\address{Institute of Discrete Mathematics and Geometry\\
TU Wien\\
Wiedner Hauptstra{\ss}e 8-10/104\\
1040 Vienna\\
Austria}
\email{matteo.raffaelli@tuwien.ac.at}
\thanks{The second author was supported by Austrian Science Fund (FWF) project F~77.}
\date{January 5, 2024}
\begin{abstract}
We study totally umbilic isometric immersions between Riemannian manifolds. First, we provide a novel characterization of the totally umbilic isometric immersions with parallel normalized mean curvature vector, i.e., those having nonzero mean curvature vector and such that the unit vector in the direction of the mean curvature vector is parallel in the normal bundle. Such characterization is based on a family of curves, called planar pseudo-geodesics, representing a natural extrinsic generalization of both geodesics and Riemannian circles: being \emph{planar}, their Cartan development in the tangent space is planar in the ordinary sense; being \emph{pseudo-geodesics}, their geodesic and normal curvatures satisfy a linear relation. We study these curves in detail and, in particular, establish their local existence and uniqueness. Moreover, in the case of codimension-one immersions, we prove the following statement: an isometric immersion $\iota \colon M \hookrightarrow Q$ is totally umbilic if and only if the extrinsic shape of every geodesic of $M$ is planar. This extends a well-known result about surfaces in $\mathbb{R}^{3}$.
\end{abstract}
\maketitle
\tableofcontents

\section{Introduction and main results}\label{intro}
Given an isometric immersion $\iota\colon M \hookrightarrow Q$ between Riemannian manifolds $M$ and $Q$, a natural problem is to describe the geometry of $\iota(M)\equiv \iota$ through the extrinsic shape of simple test curves in $M$. For example, choosing $M$-geodesics as test curves, one proves that $\iota$ is totally geodesic if and only if the extrinsic shape of every geodesic of $M$ is a geodesic of $Q$. Here, and in the rest of the paper, the \textit{extrinsic shape} of a curve $\gamma$ in $M$ is the curve $\iota \circ \gamma$. 

Another fundamental result of this sort is the well-known theorem of Nomizu and Yano~\cite{nomizu1974}, characterizing extrinsic spheres, i.e., totally umbilic submanifolds whose mean curvature vector is parallel in the normal bundle, by the property that the extrinsic shape of every circle in $M$ is a circle in $Q$. Recall that a \textit{circle} in a Riemannian manifold is a curve whose Cartan development in the tangent space is an ordinary circle; see Definition~\ref{CartanDEF}.

A concept closely related to extrinsic sphere, first studied by Chen in \cite{chen1980}, is that of totally umbilic submanifold with \emph{normalized} parallel mean curvature vector. In this case, only a unit vector field in the direction of the (nonzero) mean curvature vector is required to be parallel.

A generalization of Nomizu--Yano's theorem to this broader class of isometric immersions appeared in \cite{sugiyama2007, adachi2008}. In order to present this generalization, we need some preliminaries. Let $\gamma$ be a smooth unit-speed curve in $M$, and let $\kappa$ be its geodesic curvature, i.e., $\kappa = \langle \grad_{\gamma'} \gamma', \grad_{\gamma'} \gamma' \rangle^{1/2}$, where $\grad$ denotes the Levi-Civita connection of $M$; moreover, provided $\kappa(s) \neq 0$, let $P = \nabla_{\gamma'}\gamma'/\kappa$ in some neighborhood of $s$. Then one says that $\gamma$ has \textit{(proper) order two at the point $\gamma(s)$} if $\kappa(s) \neq 0$ and
\begin{equation*}
\begin{cases}
\mleft.\grad_{\gamma'(t)}\gamma'\mright\rvert_{t=s} = \kappa(s) P(s),\\
\mleft.\grad_{\gamma'(t)}P\mright\rvert_{t=s} = -\kappa(s) \gamma'(s).
\end{cases}
\end{equation*}

\begin{theorem}[{\cite[Theorem~4.2]{adachi2008}}] \label{AdachiTH}
The following statements are equivalent:
\begin{enumerate}[font=\upshape]
\item $\iota$ is totally umbilic and, away from geodesic points (i.e., on the open subset where the second fundamental form is nonzero), has parallel normalized mean curvature vector.
\item For every $p \in M$ and every orthonormal pair of vectors $u,v \in T_{p}M$, there exists a curve $\gamma$, defined in a neighborhood of $0$, such that  \label{AdachiTHii}
\begin{enumerate}[font=\upshape]
\item $\gamma(0)=p$, $\gamma'(0)=u$, and $\grad_{\gamma'(0)}\gamma' \rvert_{t=0}=\kappa(0)v$;
\item The extrinsic shape $\iota \circ \gamma$ of $\gamma$ has order two at $\iota(p)$;
\item $\kappa'(0)/\kappa(0) = \tilde{\kappa}'(0)/\tilde{\kappa}(0)$, where $\tilde{\kappa}$ is the geodesic curvature of $\iota \circ \gamma$.
\end{enumerate}
\end{enumerate}
\end{theorem}

It is clear that Theorem~\ref{AdachiTH} is conceptually rather different than Nomizu--Yano's classic result and more difficult to understand geometrically. Here, by selecting an appropriate family of test curves, we shall offer a simpler characterization of the same class of submanifolds. These curves are a natural extrinsic extension of both geodesics and Riemannian circles, called \emph{pseudo-geodesics}.

Pseudo-geodesics were introduced in the literature in 1950 for surfaces embedded in three-dimensional Euclidean space~\cite{wunderlich1950c, wunderlich1950d}. A unit-speed curve $\tilde{\gamma} = \iota \circ\gamma$ lying on a surface is a \textit{pseudo-geodesic} if the acceleration vector $\tilde{\gamma}''$ makes a constant angle $\theta$ with the surface normal. Note that the angle is zero precisely when $\gamma$ is a geodesic.

This definition extends straightforwardly to any Riemannian submanifold; indeed, it is easy to see that the angle $\theta$ is constant if and only if either $\gamma$ is a geodesic or else the ratio between the (signed) normal curvature $\langle \tilde{\gamma}'', N \rangle$ and the (signed) geodesic curvature $\langle \gamma'', N \times \gamma' \rangle$ is constant (Lemma~\ref{PsGeoLM}). Hence, for the classical definition to make sense in arbitrary dimension and codimension, one just needs to interpret the geodesic curvature as $\kappa$ and the normal curvature as $\tau = \langle \alpha(\tilde{\gamma}',\tilde{\gamma}'), \alpha(\tilde{\gamma}',\tilde{\gamma}') \rangle^{1/2}$, where $\alpha$ is the second fundamental form. In fact, we shall define pseudo-geodesics in any Riemannian manifold equipped with a field of vector-valued symmetric bilinear forms (Definition~\ref{PsGeoDEF}).

On the other hand, in dimension greater than two, given a point $p \in M$, a tangent vector $v \in T_{p}M$, and a constant $c>0$, the initial value problem for the pseudo-geodesic equation $\kappa = c \tau$ is underdetermined. Thus, in order to have a well-posed problem, we consider \emph{planar} pseudo-geodesics, i.e., pseudo-geodesics that---when $\kappa >0$---have order two at all points; see Definition~\ref{PlanarDEF} and Remark~\ref{PlanarRMK}. More precisely, pseudo-geodesics whose Cartan development in the tangent space of $M$ lies in a plane (Proposition~\ref{DevelopmPROP}).

Using them as test curves, we will prove our first result.

\begin{theorem} \label{MainTH}
If, for some constant $c > 0$, the extrinsic shape of every planar $c$-pseudo-geodesic of $(M,\iota^{\ast}\alpha)$ is planar, then $\iota$ is totally umbilic and, away from geodesic points, has parallel normalized mean curvature vector. Conversely, if $\iota$ is totally umbilic with parallel normalized mean curvature vector, then the extrinsic shape of every planar pseudo-geodesic of $(M,\iota^{\ast}\alpha)$ is planar.
\end{theorem}

\begin{corollary}\label{COR}
If $\iota$ is totally umbilic with parallel normalized mean curvature vector, then the extrinsic shape of every geodesic of $M$ is planar.
\end{corollary}

\begin{corollary}[{\cite[Theorem~2]{nomizu1973}}] 
If $\iota$ is a non-totally geodesic extrinsic sphere, then the extrinsic shape of every geodesic of $M$ is a circle.
\end{corollary}

\begin{remark}
If $\iota$ is a hypersurface, then the normalized mean curvature vector is automatically parallel.
\end{remark}

\begin{remark}
It is easy to see, by means of the Gauss formula, that ($c > 0$)-pseudo-geodesics satisfy the condition $\kappa' / \kappa = \tilde{\kappa}' / \tilde{\kappa}$; cf.\ Theorem~\ref{AdachiTH}\eqref{AdachiTHii}. Hence the first part of Theorem~\ref{MainTH} may be seen as a direct consequence of the local existence of planar pseudo-geodesics (Proposition~\ref{ExistencePROP}).
\end{remark}

\begin{remark}
The assumption in the first part of Theorem~\ref{MainTH} can be weakened: it is enough to assume the extrinsic shape of every planar $(c>0)$-pseudo-geodesic to be planar on any interval where the curvature is strictly positive.
\end{remark}

It is well known that if an isometric immersion takes planar curves to planar curves, then it is totally geodesic~\cite[Theorem~1]{tanabe2006}. A natural question, then, is whether it is possible for the type of immersion considered in Theorem~\ref{MainTH} to preserve the planarity of additional curves without necessarily being totally geodesic. Our next result answers this question negatively.

\begin{proposition}\label{PlanarImpliesPseudoGeodesic}
Suppose that $\iota$ is totally umbilic with parallel normalized mean curvature vector. If a curve has planar extrinsic shape, then it is a pseudo-geodesic.
\end{proposition}

An additional problem that Theorem~\ref{MainTH} leaves open is to characterize the submanifolds all of whose geodesics have planar extrinsic shape. The particular case where the ambient manifold is a space form was examined in \cite{hong1973, little1976, sakamoto1977, adachi2005, hong2008}. Here we shall give a complete solution when the codimension is one.
\begin{theorem}\label{SecondTH}
Suppose that $M$ is a hypersurface of $Q$. If the extrinsic shape of every geodesic of $M$ is planar, then $\iota$ is totally umbilic. In particular, if the extrinsic shape of every geodesic of $M$ is a circle, then $\iota$ is a non-totally geodesic extrinsic sphere.
\end{theorem}

Theorem~\ref{SecondTH} extends a classical results about surfaces in $\mathbb{R}^{3}$; see, for instance, \cite[pp.~211--212]{spivak1999}.

The paper is organized as follows. The next section presents some preliminaries, mostly for the sake of fixing relevant notation and terminology. In section~\ref{planar}, motivated by the notion of Riemannian circle, we introduce \emph{planar} Riemannian curves, thus generalizing the standard notion of planarity valid in space forms. In section~\ref{planarPG} we then define pseudo-geodesics; in particular, by restricting our attention to planar pseudo-geodesics, we establish a local existence and uniqueness result. In section~\ref{proofs} we proceed with the proofs of Theorem~\ref{MainTH}, Proposition~\ref{PlanarImpliesPseudoGeodesic}, and Theorem~\ref{SecondTH}; although, by virtue of Proposition~\ref{ExistencePROP}, the first part of Theorem~\ref{MainTH} could be obtained directly from Theorem~\ref{AdachiTH}, we give an independent proof of Theorem~\ref{MainTH}. Finally, in section~\ref{generalization} we extend Theorem~\ref{SecondTH} to submanifolds of arbitrary codimension.

As already indicated, pseudo-geodesics have a rather long history. Their study goes back to the work of Wunderlich, who defined them in the classical framework of surfaces in $\mathbb{R}^{3}$~\cite{wunderlich1950c, wunderlich1950d, wunderlich1950a, wunderlich1950b}. Pseudo-geodesics were later studied by Simon~\cite{simon1972} and, in more general settings, by Sachs~\cite{sachs1980a} and Sachs--Strommer~\cite{sachs1980b}. Interestingly, they appear in the theory of developable surfaces with creases; see \cite[p.~421]{pottmann2001} and \cite{jiang2019}. It was a surprise to discover that they also play a role in the theory of isometric immersions.

\section{Preliminaries}
In this section we recall some basic facts that are used throughout the paper.

To begin with, let $Q$ be a Riemannian manifold and $M \subset Q$ an immersed submanifold. Identifying, as customary, the tangent space of $M$ at $p$ with its image under the differential of the inclusion $M \hookrightarrow Q$, we have the orthogonal decomposition
\begin{equation*}
T_{p}Q = T_{p}M \oplus N_{p}M ,
\end{equation*}
where $N_{p}M$ is the normal space of $M$ at $p$.

Under this identification, every smooth vector field $X$ \emph{on} $M$ can be considered as a vector field $X$ \emph{along} $M$, that is, a smooth section of the ambient tangent bundle over $M$.

Let $\mathfrak{X}(M)$ and $\bar{\mathfrak{X}}(M)$ denote the sets of smooth vector fields \emph{on} and \emph{along} $M$, respectively. Let $\mathfrak{X}(M)^{\perp}$ be the set of \emph{normal} vector fields along $M$. Clearly, $\bar{\mathfrak{X}}(M) = \mathfrak{X}(M)\oplus \mathfrak{X}(M)^{\perp}$. If $X\in \mathfrak{X}(M)$ and $\bar{X} \in \bar{\mathfrak{X}}(M)$, then
\begin{equation*}
\tildegrad_{X} \bar{X} =\pitop \tildegrad_{X} \bar{X} + \piperp \tildegrad_{X} \bar{X},
\end{equation*}
where $\tildegrad$ is the Levi-Civita connection of $Q$, $\pitop$ and $\piperp$ are the orthogonal projections onto the tangent and normal bundle of $M$, and where both $X$ and $\bar{X}$ are extended arbitrarily to $Q$.

In particular, if $\bar{X} = Y \in \mathfrak{X}(M)$, then
\begin{equation*}
\tildegrad_{X}Y = \grad_{X} Y + \alpha(X,Y);
\end{equation*}
here $\grad$ is the Levi-Civita connection of $(M,\iota^{\ast}\tilde{g})$ and $\alpha$ the second fundamental form.

Similarly, if $N \in\mathfrak{X}(M)^{\perp}$, then, denoting by $A_{N}$ the shape operator of $M$ with respect to $N$ and by $\grad^{\perp}$ the normal connection of $M$,
\begin{equation*}
\tildegrad_{X} N = A_{N}(X) + \grad^{\perp}_{X}N.
\end{equation*}

The normal connection allows us to introduce a natural covariant differentiation $\nabla^{\ast}$ for the second fundamental form, as follows; see \cite[p.~231]{lee2018} for more details. Let $F$ be the smooth vector bundle over $M$ whose fiber at the point $p \in M$ is the set of bilinear maps $T_{p}M \times T_{p}M \to N_{p}M$. For any smooth section $B$ of $F$ and any $X\in \mathfrak{X}(M)$, let $\nabla^{\ast}_{X}B$ be the smooth section of $F$ given by
\begin{equation*}
\mleft(\nabla_{X}^{\ast}B\mright)(Y,Z)= \grad^{\perp}_{X}(B(Y,Z))-B(\grad_{X}Y,Z)-B(Y,\grad_{X}Z).
\end{equation*}
It is standard to prove that $\nabla^{\ast}$ is a connection in $F$.

We next turn our attention to totally umbilic submanifolds.

Given any normal vector $\eta \in NM$, we say that $M$ is \textit{umbilic in direction $\eta$} if the shape operator $A_{\eta}$ is a multiple of the identity. If $M$ is umbilic in every normal direction, then $M$ is said to be a \textit{totally umbilic submanifold} of $Q$.

One can show that $M$ is totally umbilic if and only if, for every pair of vector fields $X,Y \in \mathfrak{X}(M)$, the following relation holds between the second fundamental form and the mean curvature vector $H$ of $M$:
\begin{equation*}
\alpha(X,Y)=\langle X,Y\rangle H.
\end{equation*}

Recall that the \textit{mean curvature vector} of $M$ is the normal vector field along $M$ given by
\begin{equation*}
H= m^{-1} \tr(\alpha),
\end{equation*}
where $m = \dim M$ and $\mathrm{tr}(\alpha)$ is the trace of $\alpha$. Equivalently, in terms of a local orthonormal frame $(X_{1}, \dotsc, X_{m})$ for $TM$,
\begin{equation*}
H= m^{-1} \sum_{j =1}^{m} \alpha(X_{j},X_{j}) .
\end{equation*}

Among totally umbilic submanifolds, extrinsic spheres are particularly important. A totally umbilic submanifold is called an \textit{extrinsic sphere} if the mean curvature vector is parallel with respect to the normal connection, that is, if $\grad^{\perp}_{X}H=0$ for all $X \in \mathfrak{X}(M)$. Beware that some authors require $H$ to be nonzero.

Our main interest in this paper lies in the family of totally umbilic submanifolds with \textit{parallel normalized mean curvature vector}, which naturally generalizes that of non-totally geodesic extrinsic spheres.

Suppose that $H$ is always different from zero. Then the unit normal vector field $H/\lVert H \rvert$ is well-defined. One says that $M$ has \textit{parallel normalized mean curvature vector} if 
\begin{equation*}
\nabla^{\perp}_{X} (H/\lVert H \rvert) =0 \quad \text{for all $X \in \mathfrak{X}(M)$}.
\end{equation*}

\section{Planar curves}\label{planar}
In this section we define planar curves in a Riemannian manifold $M \equiv (M, \langle \cdot{,} \cdot \rangle)$ and extend several well-known results about circles.

\begin{definition}\label{PlanarDEF}
Let $\gamma \colon I \to M$ be a (smooth) unit-speed curve, and denote by $T$ its tangent vector. We say that $\gamma$ is \textit{planar} if there exist a unit vector field $Y$ along $\gamma$ and a function $f \colon I \to \mathbb{R}$ such that
\begin{equation}\label{PlanarSys}
\begin{cases}
\grad_{T} T = f Y,\\
\grad_{T} Y = -f T .
\end{cases}
\end{equation}
\end{definition}

\begin{remark}\label{PlanarRMK}
\leavevmode
\begin{itemize}
\item In dimension two every (unit-speed) curve is planar.
\item A geodesic is a planar curve with $f=0$.
\item A planar curve with constant $f>0$ is called a \emph{circle}~\cite{nomizu1974}.
\item If $M$ has constant sectional curvature, then a curve in $M$ is planar if and only if it lies in some two-dimensional, totally geodesic submanifold of $M$. 
\item When $\kappa > 0$, a curve $\gamma$ is planar if and only if it has order two at all points.
\end{itemize}
\end{remark}

Hence, planar curves generalize Riemannian circles. Nomizu and Yano proved that circles are precisely those curves in $M$ that satisfy the differential equation
\begin{equation*}
\grad^{2}_{T}T +\kappa^{2}T=0,
\end{equation*}
where $\kappa=\langle \grad_{T} T, \grad_{T} T \rangle^{1/2}$ is the geodesic curvature. In the case of planar curves, the following lemma holds.

\begin{lemma}[{\cite[Lemma~2.3]{kozaki2002}}] \label{PlanarLM}
Suppose that $\kappa>0$. Then $\gamma$ is planar if and only if 
\begin{equation} \label{PlanarEQ}
\kappa \grad^{2}_{T}T + \kappa^{3}T- \kappa'\grad_{T}T =0.
\end{equation}
\end{lemma}
\begin{proof}
Let $Y = \nabla_{T}T/\kappa$. If $\gamma$ is planar, then
\begin{align*}
\grad^{2}_{T}T = \grad_{T} (\kappa Y) &= \kappa \grad_{T}Y+T(\kappa)Y\\
&= -\kappa^{2}T+\kappa' Y\\
&= -\kappa^{2}T+\kappa'  \frac{\grad_{T}T}{\kappa},
\end{align*}
which implies \eqref{PlanarEQ}.

As for the converse, let us compute
\begin{equation*}
\grad_{T}Y = \grad_{T} \mleft(\frac{\grad_{T}T}{\kappa}\mright) = \frac{\grad^{2}_{T}T}{\kappa}+ \mleft(\frac{1}{\kappa}\mright)' \grad_{T}T .
\end{equation*}
In particular, if \eqref{PlanarEQ} holds, then
\begin{align*}
\grad_{T}Y &= \frac{\frac{\kappa'}{\kappa}\grad_{T}T - \kappa^{2}T}{\kappa}+ \mleft(\frac{1}{\kappa}\mright)' \grad_{T}T\\
&= - \kappa T,
\end{align*}
as desired.
\end{proof}


We now characterize planar curves through the notion of development, in the sense of Cartan. Our result extends Nomizu--Yano's \cite[Proposition~3]{nomizu1974}. The proof is conceptually the same as the one in \cite{nomizu1974}, but is nevertheless included for the reader's convenience.
\begin{definition}\label{CartanDEF}
Let $p=\gamma(u)$ be an arbitrary point in the image of $\gamma$. The \textit{Cartan development} of $\gamma$ in the tangent space $T_{p}M$ is the unique curve $\gamma^{\ast} \colon I \to T_{p}M$ such that
\begin{enumerate}
\item $(\gamma^{\ast})'(u)=T(u)$;
\item for all $t \in I$, the vector $(\gamma^{\ast})'(t)$ is parallel---in the Euclidean sense---to the parallel transport $\tau_{u}^{t}(T(t))$ of $T(t)$ from $\gamma(t)$ to $\gamma(u)$ along $\gamma$.
\end{enumerate}
\end{definition}

\begin{proposition}\label{DevelopmPROP}
A curve $\gamma$ is planar if and only if its development $\gamma^{\ast}$ in the tangent space $T_{p}M$ is a regular planar curve in the ordinary Euclidean sense.
\end{proposition}

\begin{proof}
Assume that $\gamma$ is planar so that \eqref{PlanarSys} holds. Let
\begin{equation*}
T^{\ast}(t) = \tau_{u}^{t} \,T(t) \quad \text{and} \quad Y^{\ast}(t) = \tau_{u}^{t}\, Y(t).
\end{equation*}
Since the map $\tau_{u}^{t} \colon T_{\gamma(t)}M \to T_{\gamma(u)}M$ is linear and $\tau_{u}^{t+h}=\tau_{u}^{t}\circ\tau_{t}^{t+h}$, we obtain
\begin{align*}
\mleft(T^{\ast}\mright)'(t) &= \lim_{h \to 0} \frac{T^{\ast}(t+h)-T^{\ast}(t)}{h}=  \lim_{h \to 0} \frac{\tau_{u}^{t} \mleft(\tau_{t}^{t+h}\,T(t+h) -T(t)\mright)}{h}\\
&= \tau_{u}^{t} \, \lim_{h\to 0} \frac{\tau_{t}^{t+h}\, T(t+h) -T(t)}{h}\\
&=\tau_{u}^{t} \, \nabla_{T}T (t).
\end{align*}
Likewise, we have
\begin{equation*}
\mleft(Y^{\ast}\mright)'(t)=\tau_{u}^{t}\, \nabla_{T}Y (t) .
\end{equation*}
By virtue of these identities, \eqref{PlanarSys} implies
\begin{equation} \label{PlanarSys2}
\mleft(T^{\ast}\mright)' = f Y^{\ast}  \quad \text{and} \quad \mleft(Y^{\ast}\mright)' =- f T^{\ast}.
\end{equation}
Clearly, these equations express the fact that $\gamma^{\ast}$ is a planar curve in $T_{p}M$.

Conversely, assume that the development $\gamma^{\ast}$ of $\gamma$ in $T_{p}M$ is a regular planar curve. Then there exist a unit vector field $Y^{\ast}$ along $\gamma^{\ast}$ and a smooth function $f$ that, together with $T^{\ast}=(\gamma^{\ast})'$, satisfy \eqref{PlanarSys2}. Since $\tau_{u}^{t}$ is an isomorphism between $T_{\gamma(t)}M$ and $T_{\gamma(u)}M$, we obtain \eqref{PlanarSys} from \eqref{PlanarSys2}.
\end{proof}

We conclude this section by proving a global existence theorem, which extends \cite[Theorem~1]{nomizu1974}.

\begin{theorem}
Suppose that $M$ is complete. For any orthonormal pair of vectors $x,y \in T_{p}M$ and any smooth function $f \colon \mathbb{R} \to \mathbb{R}$, there exists a unique planar curve $\gamma \colon \mathbb{R} \to M$ satisfying \eqref{PlanarSys} and such that $\gamma(0)=p$, $T(0)=x$, and $Y(0)=y$.
\end{theorem}

\begin{proof}
By the standard theory of ordinary differential equations, the problem defined by the system \eqref{PlanarSys2} and the initial condition 
\begin{equation*}
\gamma^{\ast}(0)=p, \quad T^{\ast}(0)=x , \quad Y^{\ast}(0)=y,
\end{equation*}
has a unique global solution $\gamma^{\ast}$ in $T_{p}M$. Since $M$ is complete, it follows from \cite[p.~172, Theorem~4.1]{kobayashi1963} that there is a curve in $M$ whose development in $T_{p}M$ is $\gamma^{\ast}$. By the proof of Proposition~\ref{DevelopmPROP}, this is precisely the desired curve $\gamma$.
\end{proof}

\section{Planar pseudo-geodesics}\label{planarPG}
In \cite{wunderlich1950c} Wunderlich considered a natural extrinsic generalization of a geodesic of a surface $\iota \colon S \hookrightarrow \mathbb{R}^{3}$. Noting that a curve $\gamma$ is a geodesic of $S$ if and only if the ambient acceleration $\tilde{\gamma}'' = (\iota \circ \gamma)''$ is parallel to the surface unit normal $N$, he called a curve $\tilde{\gamma}$ in $\iota(S)$ a \textit{pseudo-geodesic} if the angle $\theta$ between $\tilde{\gamma}''$ and $N$ is constant.

The next lemma characterizes pseudo-geodesics in terms of curvature.
\begin{lemma} \label{PsGeoLM}
A curve $\tilde{\gamma}$ in $\iota(S)$ is a pseudo-geodesic if and only if, for some constant $c\in \mathbb{R}$, the signed geodesic curvature $\kappa_{s}$ and the signed normal curvature $\tau_{s}$ satisfy $\kappa_{s} = c \tau_{s}$.
\end{lemma}

On the basis of this result, we reinterpret (and extend) Wunderlich's definition as follows. 

\begin{definition} \label{PsGeoDEF}
Let $M$ be a Riemannian manifold, let $\varSigma$ be a metric vector bundle of rank $n$ over $M$, and let $\sigma$ be a smooth field of symmetric bilinear forms $T_{p}M\times T_{p}M \to \varSigma_{p}$ on $M$. A unit-speed curve $\gamma \colon I \to M$ is a \textit{($c$-)pseudo-geodesic of $(M,\sigma)$} if there exist a unit vector field $Y$ along $\gamma$ and a constant $c \geq 0$ such that
\begin{equation*}
\grad_{T} T = c \lVert \sigma(T,T) \rVert Y.
\end{equation*}
\end{definition}

Now, a fundamental property of geodesics is that, for any tangent vector $x \in T_{p}M$, there exists a geodesic, defined for $\lvert t \rvert <\epsilon$ for some $\epsilon>0$, such that $\gamma_{x}(0)=p$ and $\gamma_{x}'(0)=x$. Unfortunately, unless $\dim M = 2$, one cannot expect pseudo-geodesics to enjoy the same property, as the corresponding initial value problem is underdetermined. Thus, in order to have a well-posed problem, we restrict our attention to \emph{planar} pseudo-geodesics. In that case, we can prove the following proposition.

\begin{proposition}\label{ExistencePROP}
Let $p \in M$. For any orthonormal pair of vectors $x,y \in T_{p}M$ and for any constant $c\geq 0$, there exists a unique planar pseudo-geodesic $\gamma_{x,y}$ of $(M,\sigma)$, defined for $\lvert t \rvert <\epsilon$ for some $\epsilon>0$, such that 
\begin{equation*}
\gamma_{x,y}(0)=p, \quad T_{x,y}(0)=x, \quad \nabla_{T_{x,y}}T_{x,y}(0)=c\lVert \sigma(x,x) \rVert y,
\end{equation*}
where $T_{x,y}=\gamma_{x,y}'$.
\end{proposition}

\begin{proof}
Let $\gamma \colon I \to M$ be a unit-speed curve in $M$, and let $Y$ be a unit vector field along $\gamma$. By Definition~\ref{PlanarDEF}, it is clear that $\gamma$ is a planar pseudo-geodesic of $(M,\sigma)$ if and only if there exists $c \geq 0$ such that
\begin{align*}
	\nabla_{T}T &=  c \lVert \sigma(T,T)\rVert Y,\\
	\nabla_{T}Y&=  -c \lVert\sigma(T,T) \rVert T .
\end{align*}
We shall explore how these equations look in coordinates.

To begin with, the statement of the proposition being local, we may assume that $\varSigma$ is trivial and let $(E_{1}, \dotsc, E_{n})$ be a smooth orthonormal frame for $\varSigma$. Then there are symmetric two-tensors $\sigma^{1}, \dotsc, \sigma^{n}$ on $M$ such that
\begin{equation*}
\sigma = \sigma^{1} E_{1}+\dotsb + \sigma^{n}E_{n}.
\end{equation*}
It follows that
\begin{align*}
\mleft\lVert \sigma(T,T) \mright\rVert&= \mleft< \sigma(T,T),\sigma(T,T) \mright>^{1/2}\\
&= \mleft< \sigma^{s}(T,T)E_{s},\sigma^{s}(T,T)E_{s} \mright>^{1/2}\\
&= \mleft( \mleft(\sigma^{1}(T,T)\mright)^{2} + \dotsb + \mleft(\sigma^{n}(T,T)\mright)^{2} \mright)^{1/2};
\end{align*}
here (and in the rest of this proof), the Einstein summation convention is used.

Now, suppose that $\gamma$ is contained in the domain of a smooth chart $(u^{1},\dotsc,u^{m})$ for $M$ around $p$. Expanding $T$ and $Y$ in terms of the coordinate frame $(\partial_{1} = \partial/\partial u^{1}, \dotsc, \partial_{m}=\partial/\partial u^{m})$, we obtain
\begin{align*}
T&= T^{j}\partial_{j},\\
Y&= Y^{j}\partial_{j}.
\end{align*}
Using \cite[Proposition~4.6]{lee2018}, we compute
\begin{align*}
	\grad_{T}T&= \dot{T}^{k} \partial_{k}+ T^{i} T^{j} \varGamma_{ij}^{k} \partial_{k},\\
	\grad_{T}Y&= \dot{Y}^{k} \partial_{k}+ T^{i} Y^{j} \varGamma_{ij}^{k} \partial_{k},
\end{align*}
where a dot indicates differentiation with respect to $t$ and $\varGamma_{ij}^{k}$ is assumed to be evaluated along $\gamma$.

Likewise, expressing $\sigma^{s}$ in terms of the coordinate coframe $(du^{1}, \dotsc, du^{m})$, we get
\begin{equation*}
\sigma^{s} = \sigma^{s}_{ij} \diff u^{i} \otimes  \diff u^{j}.
\end{equation*}
It follows that 
\begin{equation*}
\sigma^{s}(T(t),T(t))=\sigma^{s}_{ij}(\gamma(t))T^{i}(t) T^{j}(t).
\end{equation*}

In conclusion, $\gamma$ is a planar pseudo-geodesic of $(M,\sigma)$ if and only if, for some $c\geq 0$ and every $k =1,\dotsc,m$, the following two equations hold:
\begin{align}
\label{TEQ}\dot{T}^{k} &=  c \mleft( \mleft(\sigma^{1}_{ij} T^{i}T^{j}\mright)^{2} +\dotsb + \mleft(\sigma^{n}_{ij} T^{i}T^{j}\mright)^{2} \mright)^{1/2} Y^{k} - \varGamma_{ij}^{k} T^{i} T^{j} ,\\
\label{YEQ}\dot{Y}^{k} &=  \varGamma_{ij}^{k} T^{i} Y^{j} -c \mleft( \mleft(\sigma^{1}_{ij} T^{i}T^{j}\mright)^{2} +\dotsb + \mleft(\sigma^{n}_{ij} T^{i}T^{j}\mright)^{2} \mright)^{1/2}  T^{k}.
\end{align}
Together with $\dot{u}^{k}=T^{k}$, equations \eqref{TEQ} and \eqref{YEQ} define a system of $3m$ ordinary differential equations in the $3m$ unknown functions $(u^{k},T^{k},Y^{k})_{k=1}^{m}$, which admits a unique local solution for any initial condition $(u^{k}(0),T^{k}(0),Y^{k}(0))_{k=1}^{m}$.
\end{proof}

\section{Proofs of the main results}\label{proofs}
Here we prove the new results presented in section~\ref{intro}, starting with Theorem~\ref{MainTH}. Having already verified that the initial value problem for planar pseudo-geodesics is well-posed, the first part of Theorem~\ref{MainTH} may be obtained as a corollary of Theorem~\ref{AdachiTH}. However, we decided to include an independent proof for the benefit of the reader.

To begin with, it is useful to establish a lemma.

\begin{lemma} \label{AlphaLM}
Suppose that, for each orthonormal pair of tangent vectors $x,y$ in $T_{p}M$, either $\alpha(x,y)=0$ or $\alpha(x,x)=\alpha(y,y)=0$. Then the following conclusions hold:
\begin{enumerate}[font=\upshape]
\item\label{IT1} $\alpha(x,x)= \pm \alpha(y,y)$ for any orthonormal $x,y$ in $T_{p}M$.
\item\label{IT2} If $\alpha(x,x)=0$ for some $x$ in $T_{p}M$, then $\alpha$ vanishes at $p$.
\end{enumerate}
\end{lemma}
\begin{proof}
We first prove \eqref{IT1}. If $(x, y)$ is orthonormal, then so is $2^{-1/2}(x+y, x-y)$. Thus, assuming the hypothesis of the lemma, either $\alpha(x+y, x-y)=0$ or else $\alpha(x+y, x+y)=\alpha(x-y, x-y)=0$. It is easy to check, using bilinearity and symmetry, that the first condition implies $\alpha(x,x) = \alpha(y,y)$, whereas the second implies $\alpha(x,x) = - \alpha(y,y)$.

Now we prove \eqref{IT2}. Suppose that there is a unit vector $x_{1}$ in $T_{p}M$ such that $\alpha(x_{1},x_{1})=0$, and let $(x_{1}, \dotsc, x_{m})$ be an orthonormal basis of $T_{p}M$. Noting that the vectors $2^{-1/2}(x_{j}+x_{k})$ and $x_{h}$ are orthonormal when $h \neq j,k$, we deduce from statement~\eqref{IT1} that both $\alpha(x_{j}, x_{j})$ and $\alpha(x_{j}+x_{k}, x_{j}+x_{k})$ vanish for all $j,k =1, \dotsc, m$. Since
\begin{equation*}
\alpha(x_{j}+x_{k}, x_{j}+x_{k}) = 2\alpha(x_{j}, x_{k}),
\end{equation*}
we conclude that $\alpha(x_{j},x_{k})=0$ for all $j$ and $k$. Hence, by bilinearity, $\alpha$ vanishes at $p$.
\end{proof}

\begin{proof}[Proof of Theorem~\textup{\ref{MainTH}}]
By Lemma~\ref{PlanarLM}, the extrinsic shape of $\gamma$ is a planar curve in $Q$ precisely when
\begin{equation} \label{ExtrPlanarEQ}
\tilde{\kappa} \tildegrad^{2}_{T}T + \tilde{\kappa}^{3}T- \tilde{\kappa}' \tildegrad_{T}T =0,
\end{equation}
where $\tilde{\kappa} = \langle \tildegrad_{T} T, \tildegrad_{T} T \rangle^{1/2} >0$, and where we identified $\widetilde{T} = (\iota\circ \gamma)'$ with $T$. Since $\tildegrad_{T}T = \grad_{T}T + \alpha(T,T)$, denoting by $\tau$ the length of $\alpha(T,T)$, it follows that $\tilde{\kappa}=\sqrt{\kappa^{2}+\tau^{2}}$. 

Let $p \in M$. If all directions are asymptotic at $p$, then we have a geodesic point. On the other hand, if $x$ is a nonasymptotic vector at $p$, then, for every curve $\gamma$ such that $\gamma(0)=p$ and $\gamma'(0)=x$, there exists an open interval $(-\epsilon, \epsilon)$ such that $\tau(t) \neq0$ in $(-\epsilon, \epsilon)$. 

Assume that $x$ is not asymptotic. Then, in $(-\epsilon, \epsilon)$,
\begin{equation*}
\tilde{\kappa}' = \frac{2\kappa \kappa'+2\tau \tau'}{2\tilde{\kappa}} = \frac{\tilde{\kappa}\mleft(\kappa \kappa'+\tau \tau'\mright)}{\kappa^{2}+\tau^{2}},
\end{equation*}
so that equation \eqref{ExtrPlanarEQ} becomes
\begin{equation} \label{ExtrPlanarEQ2}
\tildegrad^{2}_{T}T + \tilde{\kappa}^{2}T- \frac{\kappa \kappa'+\tau \tau'}{\kappa^{2}+\tau^{2}} \tildegrad_{T}T =0.
\end{equation}
Moreover, by computing
\begin{align*}
\tildegrad_{T}T &= \grad_{T}T + \alpha(T,T),\\
\tildegrad^{2}_{T}T &= \grad^{2}_{T}T + \alpha(T, \grad_{T}T) + \tildegrad_{T} \alpha(T,T),
\end{align*}
we see that \eqref{ExtrPlanarEQ2} is equivalent to
\begin{equation*}
\grad^{2}_{T}T + \alpha(T, \grad_{T}T) + \tildegrad_{T} \alpha(T,T) + \tilde{\kappa}^{2}T- \frac{\kappa \kappa'+\tau \tau'}{\kappa^{2}+\tau^{2}} \mleft(\grad_{T}T + \alpha(T,T) \mright)=0.
\end{equation*}
Decomposing into tangent and normal components, we finally obtain
\begin{align}
\label{TangEQ}&\grad^{2}_{T}T + A_{\alpha(T,T)}T+ \mleft(\kappa^{2}+\tau^{2}\mright)T-\frac{\kappa \kappa'+\tau \tau'}{\kappa^{2}+\tau^{2}} \grad_{T}T=0,\\  
\label{NormEQ}&\alpha(T, \grad_{T}T) +  \grad^{\perp}_{T} \alpha(T,T) - \frac{\kappa \kappa'+\tau \tau'}{\kappa^{2}+\tau^{2}} \alpha(T,T)=0.
\end{align}

In particular, if $\kappa=c\tau$ for some $c > 0$ and $\gamma$ is planar, then \eqref{TangEQ} and \eqref{NormEQ} simplify to
\begin{align} 
&A_{\alpha(T,T)}T + \tau^{2}T=0,\notag\\  
\label{NormEQ2}&\alpha(T, \grad_{T}T) +  \grad^{\perp}_{T} \alpha(T,T) - \frac{\tau'}{\tau} \alpha(T,T)=0.
\end{align}
Using
\begin{equation*}
\mleft(\nabla_{T}^{\ast}\alpha\mright)(T,T)= \grad^{\perp}_{T}\alpha(T,T)-2\alpha(\grad_{T}T,T),
\end{equation*}
we rewrite \eqref{NormEQ2} as
\begin{equation*}
3\alpha(T, \grad_{T}T) + \mleft(\nabla_{T}^{\ast}\alpha\mright)(T,T) - \frac{\tau'}{\tau} \alpha(T,T)=0.
\end{equation*}

At $t=0$, since $(\grad_{T}T)_{t=0}=c\tau(0)Y(0)$, the last equation specializes to
\begin{equation*}
\alpha(T(0), Y(0)) =\frac{1}{3c \tau(0)} \mleft( \frac{\tau'(0)}{\tau(0)} \alpha(T(0),T(0))- \mleft(\nabla_{T}^{\ast}\alpha\mright)(T(0),T(0))\mright).
\end{equation*}
This equation implies that, given a unit vector $x \in T_{p}M$ that is not asymptotic, the value $\alpha(x,y)$ does not depend on $y \in T_{p}M$ so long as $\langle x, y \rangle =0$. In fact, since $\alpha(x,-y) = - \alpha(x,y)$, it is clear that $\alpha(x,y) = 0$ for every $x$ and $y$ such that $\langle x, y \rangle =0$. 

If, on the other hand, $x$ is asymptotic, then, for each $y$ in the orthogonal complement of $x$ in $T_{p}M$, either $\alpha(x,y)=0$ or $\alpha(y,y)=0$; indeed, if $\alpha(y,y)\neq 0$, then $\alpha(y,x)=0$ by the previous argument.

Summing up, we have shown that if \eqref{NormEQ2} holds for every nonasymptotic $x \in T_{p}M$, then, for each orthogonal pair of vectors $x,y$ in $T_{p}M$, either $\alpha(x,y)=0$ or $\alpha(x,x)=\alpha(y,y)=0$. Hence, by Lemma~\ref{AlphaLM}, if there is $x\in T_{p}M$ such that $\alpha(x,x)=0$, then $p$ is a geodesic point.

Assume that there is no such vector. It follows by continuity that there exists a neighborhood $U$ of $p$ in $M$ whose points are nongeodesic. Applying the lemma in \cite[p.~168]{nomizu1974}, we deduce that if the assumption of the first part of Theorem~\ref{MainTH} holds, then $M$ is totally umbilic in $U$ and the normalized mean curvature vector coincides with $\overline{\alpha}=\alpha(T,T)/\tau$ along $\iota \circ\gamma$. Equation \eqref{NormEQ2} therefore simplifies to
\begin{equation*}
\grad^{\perp}_{T} \alpha(T,T) - \frac{\tau'}{\tau} \alpha(T,T)=0.
\end{equation*}
Substituting $\alpha(T,T) = \tau \overline{\alpha}$, this becomes
\begin{equation*}
\grad^{\perp}_{T} \overline{\alpha}=0,
\end{equation*}
as desired.

Conversely, suppose that $\iota$ is totally umbilic and the mean curvature vector never vanishes. It follows that $\alpha(T, \grad_{T}T)= 0$ and $\overline{\alpha} = \alpha(T,T)/\tau$ coincides with the normalized mean curvature vector along $\iota \circ \gamma$. Moreover,
\begin{equation*}
\pitop \tildegrad_{T}\overline{\alpha} = \langle \tildegrad_{T}\overline{\alpha}, T \rangle T = -\langle \overline{\alpha}, \tildegrad_{T}T\rangle = -\tau T.
\end{equation*}
A straightforward computation reveals that equations \eqref{TangEQ} and \eqref{NormEQ} are now equivalent to
\begin{align}
&\grad^{2}_{T}T +\kappa^{2}T-\frac{\kappa \kappa'+\tau\tau'}{\kappa^{2}+\tau^{2}} \grad_{T}T=0,\notag\\  
\label{NormEQ3}&\tau \grad^{\perp}_{T} \overline{\alpha} + \tau' \overline{\alpha} - \tau \frac{\kappa \kappa'+\tau \tau'}{\kappa^{2}+\tau^{2}}\overline{\alpha}=0.
\end{align}

Suppose that $\gamma$ is a $c$-pseudo-geodesic. If $c=0$, then $\gamma$ is a geodesic. In that case the first equation gives an identity whereas the second reduces to $\nabla^{\perp}_{T}\overline{\alpha}=0$. On the other hand, if $c> 0$, then substituting $\tau = \kappa /c$ in the first and $\kappa = c \tau$ in the second gives
\begin{align*}
&\kappa \grad^{2}_{T}T + \kappa^{3}T - \kappa' \grad_{T}T=0,\\
& \grad^{\perp}_{T}\overline{\alpha}=0.
\end{align*}
Evidently, these two are fulfilled exactly when $\gamma$ is planar and the normalized mean curvature vector $\overline{\alpha}$ is parallel.
\end{proof}

\begin{proof}[Proof of Proposition~\textup{\ref{PlanarImpliesPseudoGeodesic}}]
Assume the hypothesis of the proposition. Then, for $\nabla^{\perp}_{T} \overline{\alpha}=0$, equation~\eqref{NormEQ3} is equivalent to
\begin{equation*}
\kappa \mleft(\tau'\kappa-\tau\kappa' \mright)=0.
\end{equation*}

Assume that $\iota\circ\gamma$ is planar. Then, on the open subset where $\kappa(t)\neq 0$,
\begin{equation*}
\mleft( \frac{\tau}{\kappa} \mright)'=0,
\end{equation*}
which implies that $\gamma$ is a pseudo-geodesic, as desired.
\end{proof}

\begin{proof}[Proof of Theorem~\textup{\ref{SecondTH}}]
Suppose that $M$ is a hypersurface of $Q$, and let $N$ be a unit normal vector field along $M$. Clearly, if $\gamma$ is a geodesic, then $\tilde{\kappa}=\tau$ and
\begin{equation*}
\tildegrad_{T}T = \alpha(T,T),
\end{equation*}
so that, away from any zero of $\tau$, equation~\eqref{ExtrPlanarEQ} reads
\begin{equation} \label{ExtrPlanarEQgeodesic}
\tau \tildegrad_{T}\alpha(T,T)  + \tau^{3}T- \tau' \alpha(T,T) =0.
\end{equation}
In particular, if $\tau$ is strictly positive, then $\alpha(T,T) =\pm\tau N$, and therefore \eqref{ExtrPlanarEQgeodesic} becomes
\begin{equation}\label{shapeoperatorEQ}
\pm\tau^{2} A(T) + \tau^{3}T =0,
\end{equation}
being $A$ the shape operator.
%

Let $\mathbb{S}_{p}^{m-1}$ be the unit sphere in $T_{p}M$, and denote by $h$ the quadratic form $h(x) = \langle \alpha(x,x), N \rangle$. Assume that equation~\eqref{shapeoperatorEQ} holds for all unit-speed geodesics originating from $p$ with nonasymptotic tangent vector. Then every such vector is an eigenvector of $A$. We first show that if $x,y \in\mathbb{S}_{p}^{m-1}$ are linearly independent and nonasymptotic, then $h(x)=\pm h(y)$. Indeed, if $h(x+y)\neq 0$, then
\begin{equation*}
h(x+y)(x+y)=A(x+y)=A(x)+A(y)=h(x)x+h(y)y,
\end{equation*}
which implies $h(x)=h(y)$. On the contrary, if $x+y$ is asymptotic, then
\begin{equation*}
0=\langle A(x+y),x+y \rangle = (h(x)+h(y))(1+\langle x, y \rangle), 
\end{equation*}
implying $h(x)=-h(y)$. We conclude that $\mathbb{S}_{p}^{m-1}$ can be decomposed as the union of the subsets
\begin{align*}
&\{x \in \mathbb{S}_{p}^{m-1} \mid \langle A(x),x\rangle =0\},\\
&\{x \in \mathbb{S}_{p}^{m-1} \mid \langle A(x),x\rangle =h(y)\},\\
&\{x \in \mathbb{S}_{p}^{m-1} \mid \langle A(x),x\rangle =-h(y)\}.
\end{align*}
It is clear that, since $x \mapsto\langle A(x),x\rangle$ is a continuous function $ \mathbb{S}_{p}^{m-1}  \to \mathbb{R}$, only one of these sets can be nonempty. This proves that $M$ is umbilic at $p$.
%
\end{proof}

\section{A generalization of Theorem~\ref{SecondTH}}\label{generalization}
	
We finally present a generalization of Theorem~\ref{SecondTH} to arbitrary codimension. To this end, let us first recall the notion of (totally) isotropic immersion, as introduced by O'Neill in \cite{oneill1965}.
\begin{definition}
Let $\iota \colon M \hookrightarrow Q$ be an isometric immersion. We say that $\iota$ is \textit{isotropic at $p\in M$} if $\langle \alpha(x,x), \alpha(x,x) \rangle =\lambda_{p} $ for all unit vectors $x\in T_{p}M$. In particular, if $\iota$ is isotropic at every point $p\in M$, then $\iota$ is called a \textit{totally isotropic} immersion. A totally isotropic immersion is \textit{constant isotropic} if $\lambda_{p}$ is constant on $M$.
\end{definition}

\begin{theorem}\label{ThirdTH}
If the extrinsic shape of every geodesic of $M$ is planar, then $\iota$ is totally isotropic. In particular, if the extrinsic shape of every geodesic of $M$ is a circle, then $\iota$ is a non-totally geodesic constant isotropic immersion.
\end{theorem}

\begin{remark}
The second part of this theorem is not new. In fact, Maeda and Sato showed that $\iota$ is a non-totally geodesic constant isotropic immersion exactly when the extrinsic shape of every geodesic of $M$ is a circle and $(\nabla_{X}^{\ast}\alpha)(X,X)=0$ for all $X \in\mathfrak{X}(M)$~\cite[Proposition~3.1]{maeda1983}. More generally, constant isotropic immersions are characterized by the property that the extrinsic shape of every circle in $M$ has constant geodesic curvature~\cite{maeda2003}.
\end{remark}
	
\begin{remark}
In a space form, if the extrinsic shape of every geodesic of $M$ is planar, then $\iota$ is constant isotropic, and thereby any such extrinsic shape is either a geodesic or a circle; see \cite{sakamoto1977}.  	
\end{remark}
	
\begin{proof}
By \cite[Lemma~1]{oneill1965}, $\iota$ is isotropic at $p$ exactly when $\langle \alpha(x,x),\alpha(x,y)\rangle=0$ for every orthonormal pair of vectors $x,y$ in $T_{p}M$. Obviously, if $\alpha(x,x)=0$, then $\langle \alpha(x,x),\alpha(x,y)\rangle=0$ for every $y$, and so we may assume that $x$ is not asymptotic. 
		
Let $\gamma$ be the unit-speed geodesic originating from $p$ with tangent vector $T(0)=x$, and let $Y$ be the parallel transport of $y$ along $\gamma$. Then
\begin{align*}
\langle \alpha(x,x),\alpha(x,y)\rangle &= \langle \alpha(T,T),\alpha(T,Y)\rangle(0)\\
&=  \langle \alpha(T,T), \tildegrad_{T}Y \rangle(0)\\
&= -\langle \tildegrad_{T}\alpha(T,T), Y \rangle(0);
\end{align*}
here we have used the Gauss formula as well as orthogonality of $Y$ and $\alpha(T,T)$.
		
We now show that $\langle \tildegrad_{T}\alpha(T,T), Y \rangle =0$ if $\iota\circ\gamma$ is planar. Indeed, since $\gamma$ is a geodesic, equation~\eqref{ExtrPlanarEQ} gives
\begin{equation}\label{ExtrPlanarEQgeodesic2}
\tau\tildegrad_{T}\alpha(T,T)=-\tau^{3}T+\tau'\alpha(T,T),
\end{equation}
which implies $\langle \tildegrad_{T}\alpha(T,T), Y \rangle =0$. This proves the first part of the theorem.
		
For the second part, assume that $\iota\circ\gamma$ is a circle, so that $\tau>0$ and $\tau'=0$. Then \eqref{ExtrPlanarEQgeodesic2} implies
\begin{equation*}
-\langle \tildegrad_{T}\alpha(T,T),T\rangle = \langle \alpha(T,T), \tildegrad_{T}T\rangle =\langle \alpha(T,T),\alpha(T,T)\rangle=\tau^{2} .
\end{equation*} 
The last equality shows that $\langle\alpha(T,T),\alpha(T,T)\rangle$ is constant along $\gamma$, and from here the statement follows easily.
\end{proof}

\section*{Acknowledgment}
We thank an anonymous referee for valuable comments.

\bibliographystyle{amsplain}
\bibliography{biblio}
\end{document}